\documentclass[12pt]{article}
\usepackage{amssymb}
\usepackage{mathrsfs}
\usepackage{amsmath}
\usepackage{amsfonts}
\usepackage{graphicx}
\usepackage{graphics}
\usepackage{tikz}
\usepackage{float}
\usepackage{pb-diagram}
\usepackage{physics}
\newtheorem{theorem}{Theorem}[section]

\newtheorem{question}{Question}

\newtheorem{corollary}{Corollary}[section]

\newtheorem{definition}{Definition}[section]

\newtheorem{lemma}{Lemma}[section]

\newtheorem{remark}{Remark}[section]

\newenvironment{proof}[1][Proof]{\noindent \textbf{#1.} }{\  \rule{0.5em}{0.5em}}
\begin{document}

\title{Perturbing eigenvalues of nonnegative centrosymmetric matrices\thanks{
Supported by Universidad Cat\'{o}lica del Norte-VRIDT 036-2020, N\'UCLEO UCN VRIDT-083-2020, Chile, ANID-Subdirecci\'on de Capital Humano/Doctorado Nacional/2021-21210056, Chile. }}
\date{}
\author{Roberto C. D\'iaz $^{a}$\thanks{
Corresponding author} Ana I. Julio$^{b}$, Yankis R. Linares$^{b}$
 \\
%EndAName
$^{a}${\small Departamento de Matem\'aticas, Universidad de La Serena} \\
{\small Cisternas 1200, La Serena, Chile.} \\
$^{b}${\small Departamento de Matem\'{a}ticas, Universidad Cat\'{o}lica del
Norte }\ \\
{\small Casilla 1280, Antofagasta, Chile.}}
\maketitle

\begin{abstract}
An $n\times n$ matrix $C$ is said to be {\it centrosymmetric} if it satisfies the relation $JCJ=C$, where $J$ is the $n\times n$ counteridentity matrix. Centrosymmetric matrices have a rich eigenstructure that has been studied extensively in the literature. Many results for centrosymmetric matrices have been generalized to wider classes of matrices that arise in a wide variety of disciplines. In this paper, we obtain interesting spectral properties for nonnegative centrosymmetric matrices. We show how to change one single eigenvalue, two or three eigenvalues of an $n\times n$ nonnegative centrosymmetric matrix without changing any of the remaining eigenvalues neither nonnegativity nor the centrosymmetric structure. Moreover, our results allow partially answer some known questions given by Guo \cite{Guo} and by Guo and Guo \cite{Guo1}. Our proofs generate algorithmic procedures that allow to compute a solution matrix.
\end{abstract}



\section{Introduction}
An $n\times n$ matrix $C=[c_{ij}]$ is said to be {\it centrosymmetric} if its entries satisfy the relation $c_{i,j}=c_{n-i+1,n-j+1}$ or equivalently if $JCJ=C$, where $J=\begin{bmatrix}\mathbf{e}_{n} \ \cdots \ \mathbf{e}_{1}\end{bmatrix}$ is the $n\times n$ counteridentity matrix. Note that $J^{\textsuperscript{T}}=J$ and $J^{\textsuperscript{2}}=I$. Centrosymmetric matrices have a rich eigenstructure that has been studied extensively in the literature (\cite{Abu, Abu2, Alan, Cent, Good, Press, Weaver}). Many results for centrosymmetric matrices have been generalized to wider classes of matrices that arise in a wide variety of disciplines such as: Differential equations, statistics, physics, communication theory, numerical analysis, engineering, etc (for more details, we refer the reader to \cite{Chu, Gene, Hill, Onper, Lin}). \\

Throughout this paper we will use the following terminology and notation: An $n\times n$ matrix $A$ with real entries is said to be {\it nonnegative} if each of its entries are nonnegative. At this case, we will write $A\geq 0$. In general, for $m\times n$ real matrices $A, B$, the notation $A\geq B$ means that the inequalities hold entrywise. $\sigma(A)$ and $\rho(A)$ will denote the spectrum and spectral radius (leading eigenvalue) of $A$, respectively. A list $\Lambda=\{\lambda_{1},\ldots,\lambda_{n}\}$ of complex numbers which is closed under complex conjugation is said to be {\it realizable} if $\Lambda=\sigma(A)$ for some $n\times n$ nonnegative matrix $A$ and we will assume in this paper that $\rho(A)=\lambda_{1}:=\max\limits_{i=1,\ldots,n}\vert \lambda_{i}\vert$, which is called {\it Perron eigenvalue}. The {\it nonnegative inverse eigenvalue problem} (NIEP) is the problem of characterising all realizable lists. In terms of $n$, the NIEP is completely solved only for $n\leq 4$. A number of sufficient conditions for the problem to have a solution are known for $n\geq 5$. For an elaborate exposition on the history of the NIEP we refer the reader to \cite{NIEP}.\\
We will denote by $\mathbb{N}_{n}$ the collection of all the lists of $n$ complex numbers which are realized by a nonnegative matrix. Similarly, denote by $\mathcal{CR}_{n} \ (\widehat{\mathcal{CR}}_{n})$ the collection of all the lists of $n$ complex numbers which are realized by a nonnegative (positive) centrosymmetric matrix. Denote by $A^{\textsuperscript{T}}$ the transpose of a matrix $A$, by $\mathbf{e}_{k}$ the $n-$dimensional vector with one in the $k-$th position and zeros elsewhere.\\

An $n\times n$ real matrix $A=[a_{ij}]$ is said to be {\it constant row sums} if all its rows sums up to the same constant $\alpha$, i.e. $\sum\limits_{j=1}^{n}a_{ij}=\alpha$, for all $i=1,\ldots,n$. The set of all $n\times n$ real matrices with constant row sums equal to $\alpha \in\mathbb{R}$ will be denoted by $\mathcal{CS}_{\alpha}$. It is clear that $\mathbf{e}^{\textsuperscript{T}}=\begin{bmatrix}1 \ \cdots \ 1\end{bmatrix}$ is an eigenvector of any matrix $A\in\mathcal{CS}_{\alpha }$, corresponding to the eigenvalue $\alpha$. \\ \ \\
This paper is devoted to the determination of spectral perturbation properties for nonnegative centrosymmetric matrices, which are motivated by some known spectral properties for general nonnegative matrices. Fiedler \cite{Fiedler} gave some interesting spectral properties for nonnegative symmetric matrices. Guo \cite{Guo} obtained some similar results for general nonnegative matrices. In what follows, we establish some perturbation results, Theorems \ref{Brauer} to \ref{Guo2} (below), which will play an important role in our work. Theorem \ref{Brauer} shows how to change one single eigenvalue of an $n\times n$ arbitrary matrix without changing any of the remaining eigenvalues, while Theorems \ref{Guo} to \ref{Guo2} show how to change two or three eigenvalues of an $n\times n$ nonnegative matrix without changing any of the remaining eigenvalues or nonnegativity:

\begin{theorem}{\rm\cite[Brauer]{Brauer}}\label{Brauer} Let $A$ be an $n\times n$ arbitrary matrix with eigenvalues $\lambda_{1},\ldots,\lambda_{n}$. Let $\mathbf{v}^{\textsuperscript{T}}=\begin{bmatrix}v_{1} \ \cdots \ v_{n}\end{bmatrix}$ be an eigenvector of $A$ corresponding to the eigenvalue $\lambda_{k}$, and let $\mathbf{q}$ be any $n-$dimensional vector. Then $A+\mathbf{vq^{\textsuperscript{T}}}$ has eigenvalues $\lambda_{1},\ldots,\lambda_{k-1},\lambda_{k}+\mathbf{v^{\textsuperscript{T}}q},\lambda_{k+1},\ldots,\lambda_{n}$.
\end{theorem}

\begin{theorem}{\rm \cite[Theorem 3.1]{Guo}}\label{Guo}
If $\Lambda=\{\lambda_{1},\lambda_{2},\lambda_{3},\ldots,\lambda_{n}\}\in \mathbb{N}_{n}$ and $\lambda_{2}$ is a real number, then for all $t\geq0$ the lists
\begin{align*}
\Lambda^{+}_{t}&=\{\lambda_{1}+t,\lambda_{2}+t,\lambda_{3},\ldots,\lambda_{n}\}\in \mathbb{N}_{n} \\
\Lambda^{-}_{t}&=\{\lambda_{1}+t,\lambda_{2}-t,\lambda_{3},\ldots,\lambda_{n}\}\in \mathbb{N}_{n}.
\end{align*}
\end{theorem}
Guo \cite{Guo} established the following two questions:
\begin{question}\label{Q1}
\rm{For any $\{\lambda_{1},\lambda_{2},\lambda_{3},\ldots,\lambda_{n}\}\in \mathbb{S}_{n}$ and $t\geq 0$, whether or not $\{\lambda_{1}+t,\lambda_{2}\pm t,\lambda_{3},\ldots,\lambda_{n}\}\in \mathbb{S}_{n}$?, where $\mathbb{S}_{n}$ is the collection of all the lists of $n$ real numbers which are realized by a nonnegative symmetric matrix.}
\end{question}
\begin{question}\label{Q2}
\rm{Do complex eigenvalues of nonnegative matrices have a property similar to Theorem \ref{Guo}?}
\end{question}
As far as we know Question \ref{Q1} remains open, while Question \ref{Q2} was partially solved by Laffey \cite{Laffey}. He obtained the following result:
\begin{theorem}{\rm \cite[Theorem 1.1]{Laffey}}\label{Laffey}
If $\Lambda=\{\lambda_{1},a+ib,a-ib,\lambda_{4},\ldots,\lambda_{n}\}\in \mathbb{N}_{n}$, where $a$ is a real number, $b>0$ and $i=\sqrt{-1}$, then for all $t\geq0$ the list
\begin{align*}
\Lambda^{-}_{t}&=\{\lambda_{1}+2t,a-t+ib,a-t-ib,\lambda_{4},\ldots,\lambda_{n}\}\in \mathbb{N}_{n}.
\end{align*}
\end{theorem}
Laffey \cite{Laffey} used a rank one perturbation to first prove Theorem \ref{Laffey} for sufficiently small $t\geq0$ and then a compactness argument is used to extend the result to all $t\geq0$. Later, Guo and Guo \cite[Theorem 1.2]{Guo1} applied a rank two perturbation and directly proved Theorem \ref{Laffey} for all $t\geq0$, establishing a constructive proof that allows easily find a nonnegative matrix to realize the perturbed list. Similarly, the authors also establish a constructive proof for the following result:
\begin{theorem}{\rm \cite[Proposition 3.1]{Guo1}}\label{Guo2}
If $\Lambda=\{\lambda_{1},a+ib,a-ib,\lambda_{4},\ldots,\lambda_{n}\}\in \mathbb{N}_{n}$, where $a$ is a real number, $b>0$ and $i=\sqrt{-1}$, then for all $t\geq 0$ the list
\begin{align*}
\Lambda^{+}_{t}&=\{\lambda_{1}+\delta t,a+t+ib,a+t-ib,\lambda_{4},\ldots,\lambda_{n}\}\in \mathbb{N}_{n}.
\end{align*}
\end{theorem}
The authors in \cite{Guo1} showed Theorem \ref{Guo2} with $\delta=4$ and they asked whether the constant $\delta$ can be improved to be $1$ or $2$. In this paper, we show that for certain cases, $\delta$ can be improved to be $2$ (Theorem \ref{PertComp}, below).
\begin{remark}\label{R1}
{\rm Let $X_{+t}$ and $X_{-t}$ be the nonnegative matrices with perturbed spectrum $\Lambda^{+}_{t}$ and $\Lambda^{-}_{t}$ respectively, constructed by the authors in \cite{Guo, Guo1} for Theorems \ref{Guo}, \ref{Laffey} and \ref{Guo2}. Then, $X_{+t}\geq Y$ and $X_{-t}\geq Y$ for all $t\geq0$, where $Y$ is the initial nonnegative matrix with unperturbed spectrum $\Lambda$.}
\end{remark}
The purpose of the paper is to answer Questions \ref{Q1} and \ref{Q2} for nonnegative centrosymmetric matrices. In particular, Theorem \ref{Guo1} (below) partially answer Question \ref{Q1} when the realizing matrix is bisymmetric (symmetric - centrosymmetric). Our paper is organized as follows{\color{blue}:} In Section \ref{s2}, we establish some well known properties for general centrosymmetric matrices. Section \ref{s3}, is devoted to obtain some spectral properties for nonnegative centrosymmetric matrices similar to some known properties for general nonnegative matrices. Finally, in Section \ref{s4}, we present our main theorems, which are spectral perturbation results that show how to change two or three eigenvalues of an $n\times n$ nonnegative centrosymmetric matrix without changing any of the remaining eigenvalues neither nonnegativity nor the centrosymmetric structure.
\section{Some properties of centrosymmetric matrices}\label{s2}
In this section we establish some well known properties of centrosymmetric matrices, which can be
found in several publications (see \cite{Alan,Cantoni}, for example).
\begin{definition}
\label{def} An $n-$dimensional vector $\mathbf{x}$ is said to be \textit{symmetric} if $J \mathbf{x}=\mathbf{x}$ and it is said to be \textit{skew-symmetric} if $J \mathbf{x}=-\mathbf{x}$.\end{definition}
\begin{lemma}{\rm\cite{Cantoni}}\label{cent}
Let $C_{1}$ and $C_{2}$ be centrosymmetric matrices and $\alpha_{1}, \alpha_{2}\in \mathbb{R}$. Then,
\ $i) \ C_{1}^{-1}$ if exists, \ $ii) \ C_{1}^{\textsuperscript{T}}$, \ $iii) \ \alpha _{1}C_{1}\pm \alpha _{2}C_{2}$, \ $iv) \ C_{1}C_{2}$, are all centrosymmetric matrices.
\end{lemma}
\begin{theorem}
{\rm\cite{Cantoni}}\label{Cen 1} Let $C$ be an $n\times n$ centrosymmetric
matrix. \\
$i)$ If $n=2m$, then $C$ can be written as $C=
\begin{bmatrix}
A & JBJ \\
B & JAJ
\end{bmatrix}
$, where $A, B$ and $J$ are $m\times m$ matrices. Moreover, $C$ is orthogonally similar to the matrix
$\begin{bmatrix}
A+JB &  \\
& A-JB
\end{bmatrix}$
and the eigenvectors corresponding to the eigenvalues of $A+JB$ can be chosen
to be symmetric, while the eigenvectors corresponding to the eigenvalues of $A-JB$ can be chosen to be skew-symmetric. Also, if $C$ is nonnegative with Perron eigenvalue $\lambda_{1}$, then $\lambda_{1}$ is the Perron eigenvalue of $A+JB$.
\newline \\
$ii)$ If $n=2m+1$, then $C$ can be written as
$C=
\begin{bmatrix}
A & \mathbf{x} & JBJ \\
\mathbf{y}^{\textsuperscript{T}} & c & \mathbf{y}^{\textsuperscript{T}}J \\
B & J\mathbf{x} & JAJ
\end{bmatrix}
$, where $A, B$ and $J$ are $m\times m$ matrices, $\mathbf{x}$ and $\mathbf{y}$
are $m-$dimensional vectors and $c$ is a real number. Moreover, $C$ is orthogonally similar to the matrix \\
$\begin{bmatrix}
c & \sqrt{2}\mathbf{y}^{\textsuperscript{T}} &  \\
\sqrt{2}\mathbf{x} & A+JB &  \\
&  & A-JB
\end{bmatrix}$ and the eigenvectors corresponding to the eigenvalues of
$\begin{bmatrix}
c & \sqrt{2}\mathbf{y}^{\textsuperscript{T}} \\
\sqrt{2}\mathbf{x} & A+JB
\end{bmatrix}$
can be chosen to be symmetric, while the eigenvectors corresponding to the eigenvalues of $A-JB$ can be chosen to
be skew-symmetric. Also, if $C$ is nonnegative with Perron eigenvalue $\lambda_{1}$, then $
\lambda_{1} $ is the Perron eigenvalue of
$\begin{bmatrix}
c & \sqrt{2}\mathbf{y}^{\textsuperscript{T}} \\
\sqrt{2}\mathbf{x} & A+JB
\end{bmatrix}$.
\end{theorem}

\section{Some spectral properties for nonnegative centrosymmetric matrices}\label{s3}
This section is devoted to obtain some spectral properties for nonnegative centrosymmetric matrices similar to some known properties for general nonnegative matrices, which will play an important role. Firstly, we establish a centrosymmetric version of a result given by Guo \cite{Guo}:
\begin{theorem}\label{SFC}
Let $C=\left[c_{ij}\right]$ be an $n\times n$ nonnegative centrosymmetric matrix with $\rho(C)=\lambda_{1}$. Then there is a nonnegative centrosymmetric matrix $\mathcal{B}$ with $\sigma(\mathcal{B})=\sigma(C)$ and satisfying $\mathcal{B}\mathbf{e}=\lambda_{1}\mathbf{e}$.
\end{theorem}
\begin{proof}
Let $\epsilon>0$ and let us write $C(\epsilon)=\left[c_{ij}+\epsilon\right]$ with $\rho(C(\epsilon))=\lambda_{1}(\epsilon)$. Then, $\lim_{\epsilon \rightarrow 0}C(\epsilon)=C$ and $\lim_{\epsilon \rightarrow 0}\lambda_{1}(\epsilon)=\lambda_{1}$. Note that $C(\epsilon)$ is a positive centrosymmetric matrix. By Perron's theorem \cite{H.J1}, there is an ${\bf x}(\epsilon)=\begin{bmatrix}x_{1}(\epsilon) \ldots x_{n}(\epsilon)\end{bmatrix}^{\textsuperscript{T}}>0$ such that $C(\epsilon){\bf x}(\epsilon)=\lambda_{1}(\epsilon){\bf x}(\epsilon)$. By Theorem \ref{Cen 1}, we can assume that ${\bf x}(\epsilon)$ is symmetric. Therefore, $D(\epsilon)=diag\{x_{1}(\epsilon), \ldots, x_{n}(\epsilon)\}$ is a nonnegative centrosymmetric diagonal matrix. Consequently, by Lemma \ref{cent}, $B(\epsilon)=D^{-1}(\epsilon)C(\epsilon)D(\epsilon)$ is a positive centrosymmetric matrix, which is similar to $C(\epsilon)$. Note also that, $B(\epsilon)\in \mathcal{CS}_{\lambda_{1}(\epsilon)}$. Hence, $\|B(\epsilon)\|_{\infty}=\lambda_{1}(\epsilon) \to \lambda_{1}$, when $\epsilon \to 0$. Thus, given $\eta = 1$ there exists a $\theta>0$ such that for all $\epsilon\in (0,\theta)$, $\| B(\epsilon) \|_{\infty} < 1 + \lambda_{1}$. So $\{ B(\epsilon) : \epsilon\in (0,\theta)\}$ \ is contained in the compact set $\left\{B\in M_{n} : \| B \|_{\infty} \leq 1+\lambda_{1} \right\}$. Therefore, we can choose a convergent subsequence $\{\epsilon_{k} \}_{k\in
\mathbb{N}} \subset (0,\theta)$ with $\lim\limits_{k\to \infty}
\epsilon_{k} = 0$ such that $\{ B(\epsilon_{k})\}_{k\in \mathbb{N}} \subset \{
B(\epsilon) : \epsilon\in (0,\theta)\}$
and $\lim\limits_{k\to\infty}B(\epsilon_{k}) = \mathcal{B}$. Note that $\mathcal{B}\geq 0$, $\mathcal{B}\mathbf{e}=\lambda_{1}\mathbf{e}$ and $\sigma(\mathcal{B})=\sigma(C)$. Finally, since
\begin{equation*}
J\mathcal{B}J=J\left(\lim\limits_{k\to\infty}B(\epsilon_{k})\right)J=\lim\limits_{k\to\infty}JB(\epsilon_{k})J=
\lim\limits_{k\to\infty}B(\epsilon_{k})=\mathcal{B},
\end{equation*}
it follows that $\mathcal{B}$ is a centrosymmetric matrix.
\end{proof}\ \\

Some important consequences of Theorem \ref{SFC} are:
\begin{corollary}\label{CP}
$i)$ If $\Lambda=\left\{\lambda_{1},\lambda _{2},\ldots,\lambda _{n}\right\}\in \mathcal{CR}_{n}$ and $\epsilon>0$, then
\begin{equation*}
\left\{\lambda_{1}+\epsilon,\lambda _{2},\ldots,\lambda _{n}\right\}\in \widehat{\mathcal{CR}}_{n}.
\end{equation*}
$ii)$ If $\left\{\lambda_{1}+\epsilon,\lambda _{2},\ldots,\lambda _{n}\right\}\in \mathcal{CR}_{n}$ for every $\epsilon>0$, then \begin{equation*}
\left\{\lambda_{1},\lambda _{2},\ldots,\lambda _{n}\right\}\in \mathcal{CR}_{n}.
\end{equation*}
$iii)$ If $\Lambda=\left\{\lambda_{1},\lambda _{2},\ldots,\lambda _{n}\right\}\in \widehat{\mathcal{CR}}_{n}$, then there exists $\epsilon>0$ such that
\begin{equation*}
\left\{\lambda_{1}-\epsilon,\lambda _{2},\ldots,\lambda _{n}\right\}\in \widehat{\mathcal{CR}}_{n}.
\end{equation*}
\end{corollary}
\begin{proof}
$i)$ Let $C$ be a nonnegative centrosymmetric matrix with spectrum $\Lambda$. By Theorem \ref{SFC}, we can assume that $C\in \mathcal{CS}_{\lambda_{1}}$. Therefore, from Lemma \ref{cent} and Theorem \ref{Brauer}, $C+\frac{\epsilon}{n}\mathbf{e}\mathbf{e}^{\textsuperscript{T}}$ is a positive centrosymmetric matrix with spectrum $\left\{\lambda_{1}+\epsilon,\lambda _{2},\ldots,\lambda _{n}\right\}$. \\ \\
$ii)$ For any $\epsilon>0$, let $C(\epsilon)$ be a nonnegative centrosymmetric matrix with spectrum $\left\{\lambda_{1}+\epsilon,\lambda_{2},\ldots,\lambda_{n}\right\}$. It follows that $\lambda_{1}+\epsilon\geq \max\limits_{i=2,\ldots, n}\{\vert \lambda_{i} \vert\}$ for any $\epsilon>0$ and thus $\lambda_{1}\geq \max\limits_{i=2,\ldots, n}\{\vert \lambda_{i} \vert\}$. By Theorem \ref{SFC}, we can assume that $C(\epsilon)\in \mathcal{CS}_{\lambda_{1}+\epsilon}$. Hence, $\|C(\epsilon)\|_{\infty}=\lambda_{1}+\epsilon \to \lambda_{1}$, when $\epsilon\to 0$. Similarly as in the proof of Theorem \ref{SFC}, we can choose a convergent subsequence
$\{C(\epsilon_{k})\}_{k\in \mathbb{N}}\subset \{C(\epsilon)\}_{\epsilon>0}$ with
$\lim\limits_{k\to \infty}\epsilon_{k} = 0$ such that
$C = \lim\limits_{k\to
\infty}C(\epsilon_{k})$
is a nonnegative centrosymmetric matrix with spectrum $\left\{\lambda_{1},\lambda _{2},\ldots,\lambda _{n}\right\}$.
\\ \\
$iii)$ Let $C=[c_{ij}]$ be a positive centrosymmetric matrix with spectrum $\Lambda$. By Theorem \ref{SFC}, we assume that $C\in \mathcal{CS}_{\lambda_{1}}$. We choose $\epsilon$, with $0< \epsilon < n\min\{c_{ij}\}$. Then, from Lemma \ref{cent} and Theorem \ref{Brauer}, $C-\frac{\epsilon}{n}\mathbf{e}\mathbf{e}^{\textsuperscript{T}}$ is a positive centrosymmetric matrix with spectrum $\left\{\lambda_{1}-\epsilon,\lambda _{2},\ldots,\lambda _{n}\right\}$.
\end{proof}\ \\

In \cite{Julio3} the authors have proved that if $\Lambda=\left\{\lambda_{1}, \mu_{1},\overline{\mu}_{1},\ldots,\mu_{m},\overline{\mu}_{m}\right\}$ has only one real positive number and $m$ pairs of conjugated complex numbers with $m$ being odd, then $\Lambda$ cannot be the spectrum of a centrosymmetric nonnegative matrix. Therefore, we focus our attention on the following set:
\begin{align}\label{Ctilde}
\widetilde{\mathbb{C}}^{\textsuperscript{n-1}} = \left\{\Gamma =\{\lambda _{2},\ldots,\lambda _{n}\} : \Gamma =\overline{\Gamma}\right\} \setminus \left\{\{\mu_{1},\overline{\mu}_{1},\ldots,\mu_{m},\overline{\mu}_{m}\} : \text{odd} \ m\right\}.
\end{align}

The following result shows that given $\Gamma=\{\lambda_{2},\ldots,\lambda_{n}\}\in \widetilde{\mathbb{C}}^{\textsuperscript{n-1}}$, it is always possible to find a $\lambda>0$ such that $\{\lambda,\lambda_{2},\ldots,\lambda_{n}\}\in \mathcal{CR}_{n}$:
\begin{theorem}\label{L4}
Given $\Gamma=\{\lambda _{2},\ldots,\lambda _{n}\}\in \widetilde{\mathbb{C}}^{\textsuperscript{n-1}}$, let
\begin{align*}
\mathscr{C}=\left\{\lambda>0 \ : \ \{\lambda,\lambda _{2},\ldots,\lambda _{n}\}\in \mathcal{CR}_{n}\right\}.
\end{align*}
Then $\mathscr{C}\neq \emptyset$.
\end{theorem}
\begin{proof}
Let $\Gamma\cup \{0\}$. If $n=2m$, we consider $\Gamma\cup \{0\}=\Gamma _{1}\cup \Gamma_{2}$, where
\begin{equation*}
\Gamma_{1}=\{0,\lambda_{2},\ldots,\lambda_{m}\},\ \ \ \ \Gamma_{2}=\{\lambda_{m+1},\ldots,\lambda_{n}\}.
\end{equation*}
In $\Gamma_{1}$, let $\lambda_{2}\geq \cdots \geq\lambda_{p-1}$ be real numbers and let $\lambda_{p}, \overline{\lambda}_{p},\ldots,\lambda_{m-1}, \overline{\lambda}_{m-1}=\lambda_{m}$ be complex nonreal numbers. In $\Gamma_{2}$, let $\lambda_{m+1}\geq \cdots \geq\lambda_{m+r}$ be real numbers and let $\lambda_{m+r+1}, \overline{\lambda}_{m+r+1},\ldots,\lambda_{n-1}, \overline{\lambda}_{n-1}=\lambda_{n}$ be complex nonreal numbers. Let us consider the $m\times m$ matrices $E_{1}=[e_{ij}]$ and $F_{1}=[f_{ij}]$, given by
{\scriptsize
\begin{align}\label{E1}
E_{1}&=
\begin{bmatrix}
0 \\
-\lambda_{2} & \lambda_{2} \\
\vdots & & \ddots \\
-\lambda_{p-1} & & &\lambda_{p-1} \\
-\Re \lambda_{p}-\Im \lambda_{p} & & & & \Re \lambda_{p} & \Im \lambda_{p} \\
-\Re \lambda_{p}+\Im \lambda_{p} & & & & -\Im \lambda_{p} & \Re \lambda_{p} \\
\vdots & & & & & &\ddots \\
-\Re \lambda_{m-1}-\Im \lambda_{m-1} & & & & & & & \Re \lambda_{m-1} & \Im \lambda_{m-1} \\
-\Re \lambda_{m-1}+\Im \lambda_{m-1} & & & & & & & -\Im \lambda_{m-1} & \Re \lambda_{m-1}
\end{bmatrix}
\end{align}
}
and
{\scriptsize
\begin{align}\label{E2}
F_{1}&=
\begin{bmatrix}
\lambda_{m+1} \\
& \lambda_{m+2} \\
& & \ddots \\
& & & \lambda_{m+r} \\
& & & & \Re \lambda_{m+r+1} & \Im \lambda_{m+r+1} \\
& & & & -\Im \lambda_{m+r+1} & \Re \lambda_{m+r+1} \\
& & & & & & \ddots \\
& & & & & & & \Re \lambda_{n-1} & \Im \lambda_{n-1} \\
& & & & & & & -\Im \lambda_{n-1} & \Re \lambda_{n-1}
\end{bmatrix},
\end{align}
}\ \\
such that $\sigma(E_{1})=\Gamma_{1}$ and $\sigma(F_{1})=\Gamma_{2}$. Note that $E_{1}\in \mathcal{CS}_{0}$. By defining the vector $\mathbf{q}^{\textsuperscript{T}}=\begin{bmatrix}q_{1} \ \cdots \ q_{m}\end{bmatrix}\geq 0$, where
\begin{equation*}
q_{j}=\max\limits_{i=1,\ldots,m}\left\{\big\vert e_{ij}+f_{ij} \big\vert, \ \big\vert e_{ij}-f_{ij} \big\vert \right\}, \ \ j=1,\ldots,m,
\end{equation*}
we obtain that $E_{1}+\mathbf{eq}^{\textsuperscript{T}}$ is a nonnegative matrix and by Theorem \ref{Brauer}, it has spectrum $\{\mathbf{e}^{\textsuperscript{T}}\mathbf{q},\lambda _{2},\ldots,\lambda _{m}\}$. Also,
\begin{align*}
(E_{1}+F_{1})+\mathbf{eq}^{\textsuperscript{T}}\geq 0 \quad \text{and} \quad (E_{1}-F_{1})+\mathbf{eq}^{\textsuperscript{T}}\geq 0.
\end{align*}
Therefore, from Theorem \ref{Cen 1},
\begin{equation*}
C=\frac{1}{2}\begin{bmatrix}
(E_{1}+F_{1})+\mathbf{eq}^{\textsuperscript{T}} & \left((E_{1}-F_{1})+\mathbf{eq}^{\textsuperscript{T}}\right)J \\ \\
J\left((E_{1}-F_{1})+\mathbf{eq}^{\textsuperscript{T}}\right) & J((E_{1}+F_{1})+\mathbf{eq}^{\textsuperscript{T}})J
\end{bmatrix},
\end{equation*}
is a nonnegative centrosymmetric matrix with spectrum $\{\mathbf{e}^{\textsuperscript{T}}\mathbf{q},\lambda _{2},\ldots,\lambda _{n}\}$. Thus, $\mathbf{e}^{\textsuperscript{T}}\mathbf{q}\in \mathscr{C}$.
\\
If $n=2m+1$, we consider $\Gamma \cup\{0\}=\Gamma _{1}\cup \Gamma_{2}$, where
\begin{equation*}
\Gamma _{1}=\{0,\lambda_{2},\ldots,\lambda_{m+1}\},\ \ \ \ \Gamma_{2}=\{\lambda_{m+2},\ldots,\lambda_{n}\}.
\end{equation*}
Similarly as in \eqref{E1} and \eqref{E2}, we can construct two matrices $E_{2}=[e_{ij}]$ and $F_{2}=[f_{ij}]$ such that $\sigma(E_{2})=\Gamma_{1}$ \ and \ $\sigma(F_{2})=\Gamma_{2}$ with $E_{2}\in \mathcal{CS}_{0}$. Let us write
\begin{equation*}
E_{2}=\begin{bmatrix}
0 & \mathbf{0}^{\textsuperscript{T}} \\
\mathbf{x} & \widehat{E}_{2}
\end{bmatrix},
\end{equation*}
where $\widehat{E}_{2}=[\widehat{e}_{ij}]$ is the $m\times m$ principal submatrix obtained to deleting the first row and first column of $E_{2}$. By defining the nonnegative vector $\mathbf{q}^{\textsuperscript{T}}=\begin{bmatrix}q_{1} \ \cdots \ q_{m+1}\end{bmatrix}=\begin{bmatrix}q_{1} \ \ {\bf \widehat{q}}^{\textsuperscript{T}}\end{bmatrix}$, where
\begin{align*}
q_{1}&=\max\limits_{i=1,\ldots,m+1}\left\{\big\vert e_{i1} \big\vert \right\} \\
q_{j}&=\max\limits_{i=1,\ldots,m+1}\left\{\big\vert \widehat{e}_{ij}+f_{ij} \big\vert, \ \big\vert \widehat{e}_{ij}-f_{ij} \big\vert \right\}, \ \ j=2,\ldots,m+1,
\end{align*}
we obtain that
\begin{equation*}
E_{2}+{\bf e}{\bf q}^{\textsuperscript{T}}=\begin{bmatrix}
0 & \mathbf{0} \\
\mathbf{x} & \widehat{E}_{2}
\end{bmatrix}+\begin{bmatrix}1 \\ {\bf e}\end{bmatrix}\begin{bmatrix} q_{1} & {\bf \widehat{q}}^{\textsuperscript{T}}\end{bmatrix}
=\begin{bmatrix}
q_{1} & \widehat{\mathbf{q}}^{\textsuperscript{T}} \\
\mathbf{x}+q_{1}{\bf e} & \widehat{E}_{2}+{\bf e}\widehat{{\bf q}}^{\textsuperscript{T}}
\end{bmatrix}
\end{equation*}
is a nonnegative matrix and by Theorem \ref{Brauer}, it has spectrum $\{\mathbf{e}^{\textsuperscript{T}}\mathbf{q},\lambda _{2},\ldots,\lambda _{m+1}\}$. Also,
\begin{align*}
(\widehat{E}_{2}+F_{2})+{\bf e}\widehat{{\bf q}}^{\textsuperscript{T}}\geq 0 \quad \text{and} \quad (\widehat{E}_{2}-F_{2})+{\bf e}\widehat{{\bf q}}^{\textsuperscript{T}}\geq 0.
\end{align*}
Therefore, from Theorem \ref{Cen 1},
\begin{equation*}
C=\frac{1}{2}\begin{bmatrix}
(\widehat{E}_{2}+F_{2})+{\bf e}\widehat{{\bf q}}^{\textsuperscript{T}} & \sqrt{2}(\mathbf{x}+q_{1}{\bf e}) & ((\widehat{E}_{2}-F_{2})+{\bf e}\widehat{{\bf q}}^{\textsuperscript{T}})J \\ \\
\sqrt{2}\widehat{\mathbf{q}}^{\textsuperscript{T}} & 2q_{1} & \sqrt{2}\widehat{\mathbf{q}}^{\textsuperscript{T}}J \\ \\
J((\widehat{E}_{2}-F_{2})+{\bf e}\widehat{{\bf q}}^{\textsuperscript{T}})& \sqrt{2}J(\mathbf{x}+q_{1}{\bf e}) & J((\widehat{E}_{2}+F_{2})+{\bf e}\widehat{{\bf q}}^{\textsuperscript{T}})J
\end{bmatrix},
\end{equation*}
is a nonnegative centrosymmetric matrix with spectrum $\{\mathbf{e}^{\textsuperscript{T}}\mathbf{q},\lambda _{2},\ldots,\lambda _{n}\}$. Thus, $\mathbf{e}^{\textsuperscript{T}}\mathbf{q}\in \mathscr{C}$.
\end{proof}\ \\

The previous result allows us to guarantee the existence of a minimum element $\lambda_{\Gamma}>0$ such that $\{\lambda_{\Gamma}\}\cup \Gamma\in \mathcal{CR}_{n}$, for each $\Gamma\in \widetilde{\mathbb{C}}^{\textsuperscript{n-1}}$. More precisely:
\begin{theorem}
Let $\Gamma=\{\lambda_{2},\ldots,\lambda_{n}\}\in \widetilde{\mathbb{C}}^{\textsuperscript{n-1}}$, where $\widetilde{\mathbb{C}}^{\textsuperscript{n-1}}$ is as in \eqref{Ctilde}. Then there exists a real number $\lambda_{\Gamma}$ with $\lambda_{\Gamma}\geq \max\limits_{i=2,\ldots, n}\{\vert \lambda_{i} \vert\}$, such that
\begin{equation*}
\{\lambda,\lambda_{2},\ldots,\lambda_{n}\}\in \mathcal{CR}_{n}
\end{equation*}
if and only if $\lambda\geq \lambda_{\Gamma}$.
\end{theorem}
\begin{proof} We split the proof into two steps: \\
{\it Step 1. Existence of $\lambda_{\Gamma}$}: By Theorem \ref{L4} we know that
\begin{equation*}
\mathscr{C}=\left\{\lambda>0 \ : \ \{\lambda,\lambda _{2},\ldots,\lambda _{n}\}\in \mathcal{CR}_{n}\right\}\neq \emptyset.
\end{equation*}
Since $\lambda \geq \max\limits_{i=2,\ldots,n} \left\{\vert \lambda_{i} \vert\right\}$ for all $\lambda\in \mathscr{C}$, $\mathscr{C}$ is a subset of $\mathbb{R}$ bounded below. Therefore, there exists $\inf\mathscr{C}$. Now, we show that $\inf\mathscr{C}\in \mathscr{C}$. In fact, we suppose that $\inf\mathscr{C}\notin \mathscr{C}$. Then $\{\inf\mathscr{C},
\lambda_{2},\ldots,\lambda_{n}\}\notin \mathcal{CR}_{n}$. From Corollary \ref{CP}, there must be at least an $\epsilon>0$ such that
\begin{align}\label{I1}
\{\inf\mathscr{C}+\epsilon,\lambda_{2},\ldots,\lambda_{n}\}\notin \mathcal{CR}_{n}.
\end{align}
{\it Assertion.} For all $\lambda\in \mathscr{C}$, $\inf\mathscr{C}+\epsilon< \lambda$: In fact, if there exists a $\lambda\in
\mathscr{C}$ such that $\lambda< \inf\mathscr{C}+\epsilon$, then choosing $r=\inf\mathscr{C}+\epsilon-\lambda > 0$ we get that $\lambda+r=\inf\mathscr{C}+\epsilon$. Since $\lambda\in \mathscr{C}$, $\{\lambda,\lambda_{2},\ldots,\lambda_{n}\}\in \mathcal{CR}_{n}$ and therefore by Corollary \ref{CP}, $\{\lambda+r,\lambda_{2},\ldots,\lambda_{n}\}\in \mathcal{CR}_{n}$. Thus, $\{\inf\mathscr{C}+\epsilon,\lambda_{2},\ldots,\lambda_{n}\}\in \mathcal{CR}_{n}$, which is a contradiction with \eqref{I1}.
\\
From above assertion, $\inf\mathscr{C}+\epsilon$ is a lower bound of $\mathscr{C}$ with $\inf\mathscr{C} < \inf\mathscr{C}+\epsilon$, which is a contradiction. Therefore, $\inf\mathscr{C}\in \mathscr{C}$. Consequently, we define $\lambda_{\Gamma}:=\min\mathscr{C}=\inf\mathscr{C}$ and thus $\lambda_{\Gamma}\geq \max\limits_{i=2,\ldots, n}\{\vert \lambda_{i} \vert\}$. \\ \\
{\it Step 2. $\{\lambda,\lambda_{2},\ldots,\lambda_{n}\}\in \mathcal{CR}_{n}$ if and only if $\lambda\geq \lambda_{\Gamma}$.} \\
{\it If:} If $\{\lambda,\lambda_{2},\ldots,\lambda_{n}\}\in \mathcal{CR}_{n}$, then $\lambda\in \mathscr{C}$. Therefore, $\lambda\geq \min\mathscr{C}=\lambda_{\Gamma}$. \\ \\
{\it Only if:} Suppose that $\lambda\geq \lambda_{\Gamma}$. If $\lambda=\lambda_{\Gamma}$, the proof follows. If $\lambda > \lambda_{\Gamma}$, then considering $\epsilon=\lambda-\lambda_{\Gamma}>0$, $\lambda_{\Gamma}+\epsilon = \lambda$. Since $\{\lambda_{\Gamma},\lambda_{2},\ldots,\lambda_{n}\}\in \mathcal{CR}_{n}$, it follows from Corollary \ref{CP} that $\{\lambda_{\Gamma}+\epsilon,\lambda_{2},\ldots,\lambda_{n}\}\in \mathcal{CR}_{n}$, this is, $\{\lambda,\lambda_{2},\ldots,\lambda_{n}\}\in \mathcal{CR}_{n}$.
\end{proof}
\section{Main results}\label{s4}
In this section we establish some interesting spectral perturbation properties for nonnegative centrosymmetric matrices. Before, we present the following lemma, which will play an important role.
\begin{lemma}\label{L5} Let $C$ be an $n\times n$ nonnegative centrosymmetric matrix with Perron eigenvalue $\lambda_{1}$.
\begin{itemize}
\item[$i)$] If $C=\begin{bmatrix}
A & JBJ \\
B & JAJ
\end{bmatrix}\in \mathcal{CS}_{\lambda_{1}}$, then $A+JB\in \mathcal{CS}_{\lambda_{1}}$.
\item[$ii)$] If $C=\begin{bmatrix}
A & \mathbf{x} & JBJ \\
\mathbf{y}^{\textsuperscript{T}} & c & \mathbf{y}^{T}J \\
B & J\mathbf{x} & JAJ%
\end{bmatrix}\in \mathcal{CS}_{\lambda_{1}}$, then $\begin{bmatrix}
c & 2\mathbf{y}^{\textsuperscript{T}} \\
\mathbf{x} & A+JB
\end{bmatrix}\in \mathcal{CS}_{\lambda_{1}}$. Moreover, $\begin{bmatrix}
c & 2\mathbf{y}^{\textsuperscript{T}} \\
\mathbf{x} & A+JB
\end{bmatrix}$ is cospectral to $\begin{bmatrix}
c & \sqrt{2}\mathbf{y}^{\textsuperscript{T}} \\
\sqrt{2}\mathbf{x} & A+JB
\end{bmatrix}$.
\end{itemize}
\end{lemma}
\begin{proof}
$i)$ Let $C=\begin{bmatrix}
A & JBJ \\
B & JAJ
\end{bmatrix}\in \mathcal{CS}_{\lambda_{1}}$. Then
\begin{align*}
C{\bf e} \ = \ \begin{bmatrix}
A & JBJ \\
B & JAJ
\end{bmatrix}\begin{bmatrix} {\bf e} \\ {\bf e} \end{bmatrix} = \begin{bmatrix}
(A+JB){\bf e} \\
(B+JA){\bf e}
\end{bmatrix}= \begin{bmatrix} \lambda_{1}{\bf e} \\ \lambda_{1}{\bf e} \end{bmatrix},
\end{align*} where $\mathbf{e}$ has been partitioned conforms to the partition of $C$.
Therefore, $(A+JB){\bf e}=\lambda_{1}{\bf e}$. \\ \\
$ii)$ \ Let $C=\begin{bmatrix}
A & \mathbf{x} & JBJ \\
\mathbf{y}^{\textsuperscript{T}} & c & \mathbf{y}^{\textsuperscript{T}}J \\
B & J\mathbf{x} & JAJ%
\end{bmatrix}\in \mathcal{CS}_{\lambda_{1}}$. Then
\begin{align*}
C{\bf e} =  \begin{bmatrix}
A & \mathbf{x} & JBJ \\
\mathbf{y}^{\textsuperscript{T}} & c & \mathbf{y}^{\textsuperscript{T}}J \\
B & J\mathbf{x} & JAJ
\end{bmatrix}\begin{bmatrix} {\bf e} \\ 1 \\ {\bf e} \end{bmatrix} =\begin{bmatrix}
(A+JB){\bf e}+{\bf x} \\
2\mathbf{y}^{\textsuperscript{T}}{\bf e}+c \\
(B+JA){\bf e}+J{\bf x}
\end{bmatrix} \ = \
\begin{bmatrix} \lambda_{1}{\bf e} \\ \lambda_{1} \\ \lambda_{1}{\bf e} \end{bmatrix}.
\end{align*}
Therefore, $\begin{bmatrix}
c & 2\mathbf{y}^{\textsuperscript{T}} \\
\mathbf{x} & A+JB
\end{bmatrix}\begin{bmatrix} 1 \\ {\bf e} \end{bmatrix}=\lambda_{1}\begin{bmatrix} 1 \\ {\bf e}\end{bmatrix}$. Moreover, if $\lambda\in \sigma\left(\begin{bmatrix}
c & 2\mathbf{y}^{\textsuperscript{T}} \\
\mathbf{x} & A+JB
\end{bmatrix}\right)$ with eigenvector ${\bf z}=\begin{bmatrix} z \\ \widetilde{{\bf z}} \end{bmatrix}$, then $\lambda\in \sigma\left(\begin{bmatrix}
c & \sqrt{2}\mathbf{y}^{\textsuperscript{T}} \\
\sqrt{2}\mathbf{x} & A+JB
\end{bmatrix}\right)$ with eigenvector ${\bf w}=\begin{bmatrix} z \\ \sqrt{2}\widetilde{{\bf z}} \end{bmatrix}$.
\end{proof}
\subsection{Perturbing real eigenvalues of nonnegative centrosymmetric matrices}
In this section, we establish a property of real eigenvalues of nonnegative centrosymmetric matrices, which is shown in the following theorem.
\begin{theorem}\label{Guo1}
If $\Lambda=\{\lambda_{1},\lambda_{2},\lambda_{3},\ldots,\lambda_{n}\}\in \mathcal{CR}_{n}$ and $\lambda_{2}$ is a real number, then for all $t\geq 0$, the lists
\begin{align*}
\Lambda^{+}_{t}&=\{\lambda_{1}+t,\lambda_{2}+t,\lambda_{3},\ldots,\lambda_{n}\}\in \mathcal{CR}_{n} \\
\Lambda^{-}_{t}&=\{\lambda_{1}+t,\lambda_{2}-t,\lambda_{3},\ldots,\lambda_{n}\}\in \mathcal{CR}_{n}.
\end{align*}
\end{theorem}
\begin{proof}
Let $t\geq 0$. We split the proof into two cases: \\
{\it Case 1, even $n$}: Since $\Lambda\in \mathcal{CR}_{n}$, let $C=\begin{bmatrix} A & JBJ \\ B & JAJ \end{bmatrix}$ be a nonnegative centrosymmetric matrix with spectrum $\Lambda$. By Theorem \ref{SFC}, we can suppose that $C \in \mathcal{CS}_{\lambda_{1}}$. Thus, by Lemma \ref{L5}, $A+JB\in \mathcal{CS}_{\lambda_{1}}$. Also, from Theorem \ref{Cen 1}, $C$ is orthogonally similar to
\begin{equation*}
\begin{bmatrix} A+JB & \\ & A-JB\end{bmatrix}.
\end{equation*}
Therefore, we have two situations: \\
$i)$ If $\lambda_{2}\in \sigma(A+JB)$, then we write without loss of generality
\begin{equation*}
\sigma(A+JB)=\{\lambda_{1},\lambda_{2},\ldots,\lambda_{\frac{n}{2}}\}
\quad \text{and} \quad
\sigma(A-JB)=\{\lambda_{\frac{n}{2}+1},\ldots,\lambda_{n}\}.
\end{equation*}
Consequently $\{\lambda_{1},\lambda_{2},\ldots,\lambda_{\frac{n}{2}}\}\in \mathbb{N}_{n}$. Thus, by Theorem \ref{Guo}, $\{\lambda_{1}+t,\lambda_{2}+t,\lambda_{3},\ldots,\lambda_{\frac{n}{2}}\}\in \mathbb{N}_{n}$. Let $X_{+t}$ be the nonnegative matrix established in Remark \ref{R1} such that $\sigma(X_{+t})=\{\lambda_{1}+t,\lambda_{2}+t, \lambda_{3},\ldots,\lambda_{\frac{n}{2}}\}$. Then,
\begin{align}\label{S4}
X_{+t} \ \geq \ A+JB.\end{align}
By \eqref{S4}, we obtain
\begin{align*}
X_{+t}+A-JB\geq 0 \quad \text{and} \quad X_{+t}-A+JB\geq 0.
\end{align*}
Therefore,
\begin{align*}
C_{+t}=
\frac{1}{2}\begin{bmatrix}
X_{+t}+A-JB & \left(X_{+t}-A+JB\right)J \\ \\
J\left(X_{+t}-A+JB\right) & J\left(X_{+t}+A-JB\right)J
\end{bmatrix}
\end{align*}
is a nonnegative centrosymmetric matrix with spectrum $\Lambda^{+}_{t}$.

Similarly, replacing $X_{+t}$ in \eqref{S4} by $X_{-t}$, with $\sigma(X_{-t})=\{\lambda_{1}+t,\lambda_{2}-t, \lambda_{3},\ldots,\lambda_{\frac{n}{2}}\}$, we obtain a nonnegative centrosymmetric matrix with spectrum $\Lambda^{-}_{t}$. \\ \\
$ii)$ If $\lambda_{2}\in \sigma(A-JB)$, then we write
\begin{equation*}
\sigma(A+JB)=\{\lambda_{1},\lambda_{3},\ldots,\lambda_{\frac{n}{2}+1}\}
\quad \text{and} \quad
\sigma(A-JB)=\{\lambda_{2},\lambda_{\frac{n}{2}+2},\ldots,\lambda_{n}\}.
\end{equation*}
Let ${\bf w}^{\textsuperscript{T}}=\begin{bmatrix}w_{1} \cdots w_{\frac{n}{2}}\end{bmatrix}\neq 0$ such that $(A-JB){\bf w}=\lambda_{2}{\bf w}$. Define by $w_{k}=\max\limits_{i=1,\ldots,\frac{n}{2}}\{w_{i}\}$ and we assume without loss of generality that $w_{k}\geq 1$. By Theorem \ref{Brauer},
$A-JB+\left(\frac{t}{w_{k}}\right){\bf w}{\bf e}_{k}^{\textsuperscript{T}}$ and $A+JB+t{\bf e}{\bf e}_{k}^{\textsuperscript{T}}$ have spectrum $\{\lambda_{2}+t,\lambda_{\frac{n}{2}+2},\ldots,\lambda_{n}\}$ and $\{\lambda_{1}+t,\lambda_{3},\ldots,\lambda_{\frac{n}{2}+1}\}$ respectively, where $A+JB+t{\bf e}{\bf e}_{k}^{\textsuperscript{T}}$ is a nonnegative matrix. Since
\begin{align*}
A+JB+t{\bf e}{\bf e}_{k}^{\textsuperscript{T}}+\left(A-JB+\left(\frac{t}{w_{k}}\right){\bf w}{\bf e}_{k}^{\textsuperscript{T}}\right) = 2A+t\left({\bf e}+\frac{{\bf w}}{w_{k}}\right){\bf e}_{k}^{\textsuperscript{T}} \geq 0
\intertext{and}
A+JB+t{\bf e}{\bf e}_{k}^{\textsuperscript{T}}-\left(A-JB+\left(\frac{t}{w_{k}}\right){\bf w}{\bf e}_{k}^{\textsuperscript{T}}\right) = 2JB+t\left({\bf e}-\frac{{\bf w}}{w_{k}}\right){\bf e}_{k}^{\textsuperscript{T}} \geq 0,
\end{align*}
we obtain that
\begin{align*}
C_{+t}=
\frac{1}{2}\begin{bmatrix}
2A+t\left({\bf e}+\frac{{\bf w}}{w_{k}}\right){\bf e}_{k}^{\textsuperscript{T}} & \left(2JB+t\left({\bf e}-\frac{{\bf w}}{w_{k}}\right){\bf e}_{k}^{\textsuperscript{T}}\right)J \\ \\
J\left(2JB+t\left({\bf e}-\frac{{\bf w}}{w_{k}}\right){\bf e}_{k}^{\textsuperscript{T}}\right) & J\left(2A+t\left({\bf e}+\frac{{\bf w}}{w_{k}}\right){\bf e}_{k}^{\textsuperscript{T}}\right)J
\end{bmatrix}
\end{align*}
is a nonnegative centrosymmetric matrix with $\sigma(C_{+t})=\Lambda^{+}_{t}$.

Similarly, replacing $A-JB+\left(\frac{t}{w_{k}}\right){\bf w}{\bf e}_{k}^{\textsuperscript{T}}$ in the previous argument by $A-JB-\left(\frac{t}{w_{k}}\right){\bf w}{\bf e}_{k}^{^{\textsuperscript{T}}}$, which has spectrum $\{\lambda_{2}-t,\lambda_{\frac{n}{2}+2},\ldots,\lambda_{n}\}$, we obtain a nonnegative centrosymmetric matrix with spectrum $\Lambda^{-}_{t}$. \\ \\
{\it Case 2, odd $n=2m+1$}: We reason in a similar way to the previous case. Since $\Lambda \in \mathcal{CR}_{n}$, let $C=\begin{bmatrix}
A & \mathbf{x} & JBJ \\
\mathbf{y}^{\textsuperscript{T}} & c & \mathbf{y}^{\textsuperscript{T}}J \\
B & J\mathbf{x} & JAJ
\end{bmatrix}\in \mathcal{CS}_{\lambda_{1}}$ be a nonnegative centrosymmetric matrix with spectrum $\Lambda$. From Theorem \ref{Cen 1} and Lemma \ref{L5}, $C$ is cospectral to
$\begin{bmatrix}
c & 2\mathbf{y}^{\textsuperscript{T}} & \\
\mathbf{x} & A+JB & \\
& & A-JB
\end{bmatrix}$, where $\begin{bmatrix}
c & 2\mathbf{y}^{\textsuperscript{T}} \\
\mathbf{x} & A+JB
\end{bmatrix}\in \mathcal{CS}_{\lambda_{1}}$. Therefore, we have two situations: \\ \\
$i)$ If $\lambda_{2}\in \sigma\left(\begin{bmatrix}
c & 2\mathbf{y}^{\textsuperscript{T}} \\
\mathbf{x} & A+JB
\end{bmatrix}\right)$, then we write
\begin{equation*}
\sigma\left(\begin{bmatrix}
c & 2\mathbf{y}^{\textsuperscript{T}} \\
\mathbf{x} & A+JB
\end{bmatrix}\right)=\{\lambda_{1},\lambda_{2},\ldots,\lambda_{m+1}\}
\quad \text{and} \quad \sigma(A-JB)=\{\lambda_{m+2},\ldots,\lambda_{n}\}.
\end{equation*}
Consequently, $\{\lambda_{1},\lambda_{2},\ldots,\lambda_{m+1}\}\in \mathbb{N}_{n}$. Thus, by Theorem \ref{Guo}, $\{\lambda_{1}+t,\lambda_{2}+t,\lambda_{3},\ldots,\lambda_{m+1}\}\in \mathbb{N}_{n}$. Let $X_{+t}$ be the nonnegative matrix established in Remark \ref{R1} such that $\sigma(X_{+t})=\{\lambda_{1}+t,\lambda_{2}+t, \lambda_{3},\ldots,\lambda_{m+1}\}$. Then,
\begin{align}\label{I4}
X_{+t} \ \geq \ \begin{bmatrix}
c & 2\mathbf{y}^{\textsuperscript{T}} \\
\mathbf{x} & A+JB
\end{bmatrix}.\end{align}
Let us partition $X_{+t}=\begin{bmatrix}
c_{+t} & \mathbf{y}^{^{\textsuperscript{T}}}_{+t} \\
\mathbf{x}_{+t} & \widehat{X}_{+t}
\end{bmatrix}$, where $\widehat{X}_{+t}$ is the $m\times m$ principal submatrix obtained to deleting the first row and first column of $X_{+t}$.
By \eqref{I4},
\begin{align*}
\widehat{X}_{+t}+A-JB\geq 0 \quad \text{and} \quad \widehat{X}_{+t}-A+JB\geq 0.
\end{align*}
Therefore,
\begin{align*}
C_{+t}=
\frac{1}{2}\begin{bmatrix}
\widehat{X}_{+t}+A-JB & 2\mathbf{x}_{+t} & \left(\widehat{X}_{+t}-A+JB\right)J \\ \\
\mathbf{y}_{+t}^{\textsuperscript{T}} & 2c_{+t} & \mathbf{y}_{+t}^{\textsuperscript{T}}J \\ \\
J\left(\widehat{X}_{+t}-A+JB\right) & 2J\mathbf{x}_{+t} & J\left(\widehat{X}_{+t}+A-JB\right)J
\end{bmatrix}.
\end{align*}
is a nonnegative centrosymmetric matrix with spectrum $\Lambda^{+}_{t}$. \\

Similarly, replacing $X_{+t}=\begin{bmatrix}
c_{+t} & \mathbf{y}^{\textsuperscript{T}}_{+t} \\
\mathbf{x}_{+t} & \widehat{X}_{+t}
\end{bmatrix}$ in \eqref{I4} by $X_{-t}=\begin{bmatrix}
c_{-t} & \mathbf{y}^{\textsuperscript{T}}_{-t} \\
\mathbf{x}_{-t} & \widehat{X}_{-t}
\end{bmatrix}$ with $\sigma(X_{-t})=\{\lambda_{1}+t,\lambda_{2}-t, \lambda_{3},\ldots,\lambda_{m+1}\}$, we obtain a nonnegative centrosymmetric matrix with spectrum $\Lambda^{-}_{t}$. \\ \\
$ii)$ If $\lambda_{2}\in \sigma(A-JB)$, then we write
\begin{equation*}
\sigma\left(\begin{bmatrix}
c & 2\mathbf{y}^{\textsuperscript{T}} \\
\mathbf{x} & A+JB
\end{bmatrix}\right)=\{\lambda_{1},\lambda_{3},\ldots,\lambda_{m+2}\} \ \text{and} \ \sigma(A-JB)=\{\lambda_{2},\lambda_{m+3},\ldots,\lambda_{n}\}.
\end{equation*}
Let ${\bf w}^{^{\textsuperscript{T}}}=\begin{bmatrix}w_{1} \cdots w_{m}\end{bmatrix}\neq 0$ such that $(A-JB){\bf w}=\lambda_{2}{\bf w}$. We define by $w_{k}=\max\limits_{i=1,\ldots,m}\{w_{i}\}$ and we assume without loss of generality that $w_{k}\geq 1$. By Theorem \ref{Brauer},
$A-JB+\left(\frac{t}{w_{k}}\right){\bf w}{\bf e}_{k}^{^{\textsuperscript{T}}}$ and
\begin{align*}
\begin{bmatrix}
c & 2\mathbf{y}^{\textsuperscript{T}} \\
\mathbf{x} & A+JB
\end{bmatrix}+t\begin{bmatrix} 1 \\ {\bf e} \end{bmatrix}\begin{bmatrix} 0 & {\bf e}_{k}^{\textsuperscript{T}}\end{bmatrix}=\begin{bmatrix}
c & 2\mathbf{y}^{\textsuperscript{T}}+t{\bf e}_{k}^{\textsuperscript{T}} \\
\mathbf{x} & A+JB+t{\bf e}{\bf e}_{k}^{\textsuperscript{T}}
\end{bmatrix}
\end{align*}
have spectrum $\{\lambda_{2}+t,\lambda_{m+3},\ldots,\lambda_{n}\}$ and $\{\lambda_{1}+t,\lambda_{3},\ldots,\lambda_{m+2}\}$ respectively, where $\begin{bmatrix}
c & 2\mathbf{y}^{\textsuperscript{T}}+t{\bf e}_{k}^{\textsuperscript{T}} \\
\mathbf{x} & A+JB+t{\bf e}{\bf e}_{k}^{\textsuperscript{T}}
\end{bmatrix}$ is a nonnegative matrix. Since
\begin{align*}
A+JB+t{\bf e}{\bf e}_{k}^{\textsuperscript{T}}+\left(A-JB+\left(\frac{t}{w_{k}}\right){\bf w}{\bf e}_{k}^{\textsuperscript{T}}\right) = 2A+t\left({\bf e}+\frac{{\bf w}}{w_{k}}\right){\bf e}_{k}^{\textsuperscript{T}} \geq 0
\intertext{and}
A+JB+t{\bf e}{\bf e}_{k}^{\textsuperscript{T}}-\left(A-JB+\left(\frac{t}{w_{k}}\right){\bf w}{\bf e}_{k}^{\textsuperscript{T}}\right) = 2JB+t\left({\bf e}-\frac{{\bf w}}{w_{k}}\right){\bf e}_{k}^{\textsuperscript{T}} \geq 0,
\end{align*}
then
\begin{align*}
C_{+t}=
\frac{1}{2}\begin{bmatrix}
2A+t\left({\bf e}+\frac{{\bf w}}{w_{k}}\right){\bf e}_{k}^{\textsuperscript{T}} & 2\mathbf{x} & \left(2JB+t\left({\bf e}-\frac{{\bf w}}{w_{k}}\right){\bf e}_{k}^{\textsuperscript{T}}\right)J
\\ \\
2\mathbf{y}^{\textsuperscript{T}}+t{\bf e}_{k}^{\textsuperscript{T}} & 2c & (2\mathbf{y}^{\textsuperscript{T}}+t{\bf e}_{k}^{\textsuperscript{T}})J \\ \\
J\left(2JB+t\left({\bf e}-\frac{{\bf w}}{w_{k}}\right){\bf e}_{k}^{\textsuperscript{T}}\right) & 2J\mathbf{x} & J\left(2A+t\left({\bf e}+\frac{{\bf w}}{w_{k}}\right){\bf e}_{k}^{\textsuperscript{T}}\right)J
\end{bmatrix}
\end{align*}
is a nonnegative centrosymmetric matrix with spectrum $\Lambda^{+}_{t}$. \\

Similarly, replacing $A-JB+\left(\frac{t}{w_{k}}\right){\bf w}{\bf e}_{k}^{\textsuperscript{T}}$ in the previous argument by $A-JB-\left(\frac{t}{w_{k}}\right){\bf w}{\bf e}_{k}^{\textsuperscript{T}}$, which has spectrum $\{\lambda_{2}-t,\lambda_{m+3},\ldots,\lambda_{n}\}$, we obtain a nonnegative centrosymmetric matrix with spectrum $\Lambda^{-}_{t}$.
\end{proof}
\subsection{Perturbing non-real eigenvalues of nonnegative centrosymmetric matrices}
Now, we establish a property of non-real eigenvalues of nonnegative centrosymmetric matrices, which is shown in the following theorem.
\begin{theorem}\label{PertComp}
If $\Lambda=\{\lambda_{1},a+ib,a-ib,\lambda_{4},\ldots,\lambda_{n}\}\in \mathcal{CR}_{n}$, where $a$ is a real number, $b>0$ and $i=\sqrt{-1}$, then for all $t\geq 0$, the lists
\begin{align*}
\Lambda^{-}_{t}&=\{\lambda_{1}+2t,a-t+ib,a-t-ib,\lambda_{4},\ldots,\lambda_{n}\}\in \mathcal{CR}_{n} \\ \\
\Lambda^{+}_{t}&=\{\lambda_{1}+\delta t,a+t+ib,a+t-ib,\lambda_{4},\ldots,\lambda_{n}\}\in \mathcal{CR}_{n},
\end{align*}
where $\delta\in \{2,4\}$.
\end{theorem}
\begin{proof}
Let $t\geq 0$. We use a similar argument as in Theorem \ref{Guo1} and we split the proof into two cases: \\
{\it Case 1, even $n$}: Let $C=\begin{bmatrix} A & JBJ \\ B & JAJ \end{bmatrix}\in \mathcal{CS}_{\lambda_{1}}$ be a nonnegative centrosymmetric matrix with spectrum $\Lambda$. We know that $C$ is orthogonally similar to
\begin{equation*}
\begin{bmatrix} A+JB & \\ & A-JB\end{bmatrix},
\end{equation*}
where $A+JB\in \mathcal{CS}_{\lambda_{1}}$. Therefore, we have two situations: \\
$i)$ If $a\pm ib\in \sigma(A+JB)$, then we write
\begin{equation*}
\sigma(A+JB)=\{\lambda_{1},a+ib,a-ib,\lambda_{4},\ldots,\lambda_{\frac{n}{2}}\}
\ \text{and} \ \sigma(A-JB)=\{\lambda_{\frac{n}{2}+1},\ldots,\lambda_{n}\}.
\end{equation*}
At this situation, to show that $\Lambda^{-}_{t}\in \mathcal{CR}_{n}$ and $\Lambda^{+}_{t}\in \mathcal{CR}_{n}$ (being $\delta=4$), we proceed of a similar way as in Theorem \ref{Guo1}, {\it Case 1, even $n$}, when $\lambda_{2}\in \sigma(A+JB)$, taking into account Theorems \ref{Laffey} and \ref{Guo2}. \\ \\
$ii)$ If $a\pm ib\in \sigma(A-JB)$, then we write
\begin{equation*}
\sigma(A+JB)=\{\lambda_{1},\lambda_{4},\ldots,\lambda_{\frac{n}{2}+2}\}
\quad \text{and} \quad \sigma(A-JB)=\{a+ib,a-ib,\lambda_{\frac{n}{2}+3},\ldots,\lambda_{n}\}.
\end{equation*}
We follow the ideas of the proof of Theorem 1.2 in \cite{Guo1}. Let us write the Jordan canonical form of $A-JB$ as follows:
\begin{equation*}
\Psi=\begin{bmatrix}
a & b \\
-b & a & & \ast \\
& & \lambda_{\frac{n}{2}+3} \\
& & & \ddots \\
& & & & \lambda_{n}
\end{bmatrix}
\end{equation*}
Let $P=\begin{bmatrix}{\bf u} & {\bf v} & {\bf w_{3}} & \cdots & {\bf w_{\frac{n}{2}}}\end{bmatrix}$ be an $\frac{n}{2}\times \frac{n}{2}$ non-singular real matrix such that $P\Psi P^{-1}=A-JB$, where
\begin{equation*}
{\bf u}^{\textsuperscript{T}}=\begin{bmatrix}u_{1} \ \cdots \ u_{\frac{n}{2}}\end{bmatrix} \quad \text{and} \quad {\bf v}^{\textsuperscript{T}}=\begin{bmatrix}v_{1} \ \cdots \ v_{\frac{n}{2}}\end{bmatrix}
\end{equation*}
are real vectors such that ${\bf u}\pm i{\bf v}$ are eigenvectors of $A-JB$ corresponding to the eigenvalues $a\pm ib$. Now, we consider
\begin{equation*}
det(i,j)=\begin{vmatrix}u_{i} & v_{i} \\ u_{j} & v_{j}\end{vmatrix}=u_{i}v_{j}-u_{j}v_{i}, \quad \text{for any} \quad 1\leq i, j\leq \frac{n}{2}.
\end{equation*}
We can assume, without loss of generality that
\begin{equation*}
\Delta=det(1,2)=\max\limits_{1\leq i, j\leq \frac{n}{2}}det(i,j).
\end{equation*}
Since $P$ is non-singular, $\Delta>0$. Let
\begin{align*}
X^{\textsuperscript{T}}&=\begin{bmatrix}x_{1} & x_{2} & 0 & \cdots & 0\end{bmatrix}P=\begin{bmatrix}t & 0 & \ast & \cdots & \ast\end{bmatrix}, \\ \\
Y^{\textsuperscript{T}}&=\begin{bmatrix}y_{1} & y_{2} & 0 & \cdots & 0\end{bmatrix}P=\begin{bmatrix}0 & t & \ast & \cdots & \ast\end{bmatrix},
\end{align*}
where
\begin{align*}
x_{1}=\frac{t}{\Delta}v_{2}, \quad \quad x_{2}=-\frac{t}{\Delta}v_{1}, \quad \quad y_{1}=-\frac{t}{\Delta}u_{2}, \quad \quad y_{2}=\frac{t}{\Delta}u_{1}.
\end{align*}
Then,
\begin{align}\label{Mat}
P({\bf e}_{1}X^{\textsuperscript{T}}+{\bf e}_{2}Y^{\textsuperscript{T}})P^{-1} =
\begin{bmatrix}
u_{1}x_{1}+v_{1}y_{1} & u_{1}x_{2}+v_{1}y_{2} & 0 & \cdots & 0 \\
u_{2}x_{1}+v_{2}y_{1} & u_{2}x_{2}+v_{2}y_{2} & 0 & \cdots & 0 \\
\vdots & \vdots & \vdots & & \vdots \\
u_{\frac{n}{2}}x_{1}+v_{\frac{n}{2}}y_{1} & u_{\frac{n}{2}}x_{2}+v_{\frac{n}{2}}y_{2} & 0 & \cdots & 0
\end{bmatrix}.
\end{align}
Therefore,
\begin{equation*}
P(\Psi-{\bf e}_{1}X^{\textsuperscript{T}}-{\bf e}_{2}Y^{\textsuperscript{T}})P^{-1}=A-JB-P({\bf e}_{1}X^{\textsuperscript{T}}+{\bf e}_{2}Y^{\textsuperscript{T}})P^{-1},
\end{equation*}
has spectrum $\{a-t+ib,a-t-ib,\lambda_{\frac{n}{2}+3},\ldots,\lambda_{n}\}$. Note also that in \eqref{Mat} we have
{\footnotesize
\begin{align}\label{Imp1}
\begin{cases}
\vert u_{i}x_{1}+v_{i}y_{1}\vert &= \frac{t}{\Delta}\vert u_{i}v_{2}-u_{2}v_{i}\vert = \frac{t}{\Delta}\vert det(i,2)\vert \leq \frac{t}{\Delta}\Delta = t, \\ \\
\vert u_{i}x_{2}+v_{i}y_{2}\vert &= \frac{t}{\Delta}\vert u_{i}v_{1}-u_{1}v_{i}\vert = \frac{t}{\Delta}\vert det(i,1)\vert \leq \frac{t}{\Delta}\Delta = t, \quad i=1,\ldots,\frac{n}{2}.
\end{cases}
\end{align}
}
On the other hand, $A+JB+t{\bf e}({\bf e}_{1}^{\textsuperscript{T}}+{\bf e}_{2}^{\textsuperscript{T}})$ is a nonnegative matrix with spectrum $\{\lambda_{1}+2t,\lambda_{4},\ldots,\lambda_{\frac{n}{2}+2}\}$. From \eqref{Mat} and \eqref{Imp1}
\begin{align*}
A+JB+t{\bf e}({\bf e}_{1}^{\textsuperscript{T}}+{\bf e}_{2}^{\textsuperscript{T}})&+\left(A-JB-P({\bf e}_{1}X^{\textsuperscript{T}}+{\bf e}_{2}Y^{\textsuperscript{T}})P^{-1}\right) = \\ \\
&= 2A+t{\bf e}({\bf e}_{1}^{\textsuperscript{T}}+{\bf e}_{2}^{\textsuperscript{T}})-P({\bf e}_{1}X^{\textsuperscript{T}}+{\bf e}_{2}Y^{\textsuperscript{T}})P^{-1}\geq 0
\end{align*}
and
\begin{align*}
A+JB+t{\bf e}({\bf e}_{1}^{\textsuperscript{T}}+{\bf e}_{2}^{\textsuperscript{T}})-&\left(A-JB-P({\bf e}_{1}X^{\textsuperscript{T}}+{\bf e}_{2}Y^{\textsuperscript{T}})P^{-1}\right) = \\ \\ &= 2JB+t{\bf e}({\bf e}_{1}^{\textsuperscript{T}}+{\bf e}_{2}^{\textsuperscript{T}})+P({\bf e}_{1}X^{\textsuperscript{T}}+{\bf e}_{2}Y^{\textsuperscript{T}})P^{-1}\geq 0,
\end{align*}
we obtain that
{\scriptsize
\begin{align*}
C_{-t}=
\frac{1}{2}\begin{bmatrix}
2A+t{\bf e}({\bf e}_{1}^{\textsuperscript{T}}+{\bf e}_{2}^{\textsuperscript{T}})-P({\bf e}_{1}X^{\textsuperscript{T}}+{\bf e}_{2}Y^{\textsuperscript{T}})P^{-1} & \left(2JB+t{\bf e}({\bf e}_{1}^{\textsuperscript{T}}+{\bf e}_{2}^{\textsuperscript{T}})+P({\bf e}_{1}X^{\textsuperscript{T}}+{\bf e}_{2}Y^{\textsuperscript{T}})P^{-1}\right)J \\ \\
J\left(2JB+t{\bf e}({\bf e}_{1}^{\textsuperscript{T}}+{\bf e}_{2}^{\textsuperscript{T}})+P({\bf e}_{1}X^{\textsuperscript{T}}+{\bf e}_{2}Y^{\textsuperscript{T}})P^{-1}\right) & J\left(2A+t{\bf e}({\bf e}_{1}^{\textsuperscript{T}}+{\bf e}_{2}^{\textsuperscript{T}})-P({\bf e}_{1}X^{\textsuperscript{T}}+{\bf e}_{2}Y^{\textsuperscript{T}})P^{-1}\right)J
\end{bmatrix}
\end{align*}
}
is a nonnegative centrosymmetric matrix with spectrum $\Lambda^{-}_{t}$. \\

Similarly, replacing $A-JB-P({\bf e}_{1}X^{\textsuperscript{T}}+{\bf e}_{2}Y^{\textsuperscript{T}})P^{-1}$ in the previous argument by $A-JB+P({\bf e}_{1}X^{\textsuperscript{T}}+{\bf e}_{2}Y^{\textsuperscript{T}})P^{-1}$, which has spectrum $\{a+t+ib,a+t-ib,\lambda_{\frac{n}{2}+3},\ldots,\lambda_{n}\}$, we obtain a nonnegative centrosymmetric matrix with spectrum $\Lambda^{+}_{t}$ (being $\delta=2$).
\\ \\
{\it Case 2, odd $n=2m+1$}: We reason in a similar way to the previous case. Since $\Lambda \in \mathcal{CR}_{n}$, let $C=\begin{bmatrix}
A & \mathbf{x} & JBJ \\
\mathbf{y}^{\textsuperscript{T}} & c & \mathbf{y}^{\textsuperscript{T}}J \\
B & J\mathbf{x} & JAJ
\end{bmatrix}\in \mathcal{CS}_{\lambda_{1}}$ be a nonnegative centrosymmetric matrix with spectrum $\Lambda$. From Theorem \ref{Cen 1} and Lemma \ref{L5}, $C$ is cospectral to
$\begin{bmatrix}
c & 2\mathbf{y}^{\textsuperscript{T}} & \\
\mathbf{x} & A+JB & \\
& & A-JB
\end{bmatrix}$, where $\begin{bmatrix}
c & 2\mathbf{y}^{\textsuperscript{T}} \\
\mathbf{x} & A+JB
\end{bmatrix}\in \mathcal{CS}_{\lambda_{1}}$. Therefore, we have two situations: \\ \\
$i)$ If $a\pm ib\in \sigma\left(\begin{bmatrix}
c & 2\mathbf{y}^{\textsuperscript{T}} \\
\mathbf{x} & A+JB
\end{bmatrix}\right)$, then we write
\begin{align*}
\sigma\left(\begin{bmatrix}
c & 2\mathbf{y}^{\textsuperscript{T}} \\
\mathbf{x} & A+JB
\end{bmatrix}\right)=\{\lambda_{1},a+ib,a-ib,\lambda_{4},\ldots,\lambda_{m+1}\},
\end{align*}
and $\sigma(A-JB)=\{\lambda_{m+2},\ldots,\lambda_{n}\}$. To show that $\Lambda^{-}_{t}\in \mathcal{CR}_{n}$ and $\Lambda^{+}_{t}\in \mathcal{CR}_{n}$ (being $\delta=4$), we proceed of a similar way as in Theorem \ref{Guo1}, {\it Case 2, odd $n$}, when $\lambda_{2}\in \sigma\left(\begin{bmatrix}
c & 2\mathbf{y}^{\textsuperscript{T}} \\
\mathbf{x} & A+JB
\end{bmatrix}\right)$, taking into account Theorems \ref{Laffey} and \ref{Guo2}. \\ \\
$ii)$ If $a\pm ib\in \sigma(A-JB)$, then we write
\begin{equation*}
\sigma\left(\begin{bmatrix}
c & 2\mathbf{y}^{\textsuperscript{T}} \\
\mathbf{x} & A+JB
\end{bmatrix}\right)=\{\lambda_{1},\lambda_{4},\ldots,\lambda_{m+3}\}
\end{equation*}
and $\sigma(A-JB)=\{a+ib,a-ib,\lambda_{m+4},\ldots,\lambda_{n}\}$. Similarly as in Case 1, even $n$, when $a\pm ib\in A-JB$, we can obtain the matrix
\begin{equation*}
A-JB-P({\bf e}_{1}X^{\textsuperscript{T}}+{\bf e}_{2}Y^{\textsuperscript{T}})P^{-1},
\end{equation*}
with spectrum $\{a-t+ib,a-t-ib,\lambda_{m+4},\ldots,\lambda_{n}\}$, where $P({\bf e}_{1}X^{\textsuperscript{T}}+{\bf e}_{2}Y^{\textsuperscript{T}})P^{-1}$ is as in \eqref{Mat}, when replacing $\frac{n}{2}$ by $m$ and its entries satisfy \eqref{Imp1}. On the other hand,
\begin{align*}
\begin{bmatrix}
c & 2\mathbf{y}^{\textsuperscript{T}} \\
\mathbf{x} & A+JB
\end{bmatrix}+t\begin{bmatrix} 1 \\ {\bf e} \end{bmatrix}\begin{bmatrix} 0 & {\bf e}_{1}^{\textsuperscript{T}}+{\bf e}_{2}^{\textsuperscript{T}}\end{bmatrix}=\begin{bmatrix}
c & 2\mathbf{y}^{\textsuperscript{T}}+t({\bf e}_{1}^{\textsuperscript{T}}+{\bf e}_{2}^{\textsuperscript{T}}) \\
\mathbf{x} & A+JB+t{\bf e}({\bf e}_{1}^{\textsuperscript{T}}+{\bf e}_{2}^{\textsuperscript{T}})
\end{bmatrix}
\end{align*}
is a nonnegative matrix with spectrum $\{\lambda_{1}+2t,\lambda_{4},\ldots,\lambda_{m+3}\}$. Since
\begin{align*}
A+JB+t{\bf e}({\bf e}_{1}^{\textsuperscript{T}}+{\bf e}_{2}^{\textsuperscript{T}})&+\left(A-JB-P({\bf e}_{1}X^{\textsuperscript{T}}+{\bf e}_{2}Y^{\textsuperscript{T}})P^{-1}\right) = \\ \\
&= 2A+t{\bf e}({\bf e}_{1}^{\textsuperscript{T}}+{\bf e}_{2}^{\textsuperscript{T}})-P({\bf e}_{1}X^{\textsuperscript{T}}+{\bf e}_{2}Y^{\textsuperscript{T}})P^{-1}\geq 0
\end{align*}
and
\begin{align*}
A+JB+t{\bf e}({\bf e}_{1}^{\textsuperscript{T}}+{\bf e}_{2}^{\textsuperscript{T}})-&\left(A-JB-P({\bf e}_{1}X^{\textsuperscript{T}}+{\bf e}_{2}Y^{\textsuperscript{T}})P^{-1}\right) = \\ \\ &= 2JB+t{\bf e}({\bf e}_{1}^{\textsuperscript{T}}+{\bf e}_{2}^{\textsuperscript{T}})+P({\bf e}_{1}X^{\textsuperscript{T}}+{\bf e}_{2}Y^{\textsuperscript{T}})P^{-1}\geq 0,
\end{align*}
we obtain that
{\scriptsize
\begin{align*}
C_{-t}=
\frac{1}{2}\begin{bmatrix}
2A+t{\bf e}({\bf e}_{1}^{\textsuperscript{T}}+{\bf e}_{2}^{\textsuperscript{T}})-P({\bf e}_{1}X^{\textsuperscript{T}}+{\bf e}_{2}Y^{\textsuperscript{T}})P^{-1} & 2\mathbf{x} &  \left(2JB+t{\bf e}({\bf e}_{1}^{\textsuperscript{T}}+{\bf e}_{2}^{T})+P({\bf e}_{1}X^{\textsuperscript{T}}+{\bf e}_{2}Y^{\textsuperscript{T}})P^{-1}\right)J \\ \\
2\mathbf{y}^{\textsuperscript{T}}+t{\bf e}_{k}^{\textsuperscript{T}} & 2c & (2\mathbf{y}^{\textsuperscript{T}}+t{\bf e}_{k}^{\textsuperscript{T}})J \\ \\
J\left(2JB+t{\bf e}({\bf e}_{1}^{\textsuperscript{T}}+{\bf e}_{2}^{\textsuperscript{T}})+P({\bf e}_{1}X^{\textsuperscript{T}}+{\bf e}_{2}Y^{\textsuperscript{T}})P^{-1}\right) & 2J\mathbf{x} & J\left(2A+t{\bf e}({\bf e}_{1}^{\textsuperscript{T}}+{\bf e}_{2}^{\textsuperscript{T}})-P({\bf e}_{1}X^{\textsuperscript{T}}+{\bf e}_{2}Y^{\textsuperscript{T}})P^{-1}\right)J
\end{bmatrix}
\end{align*}
}
is a nonnegative centrosymmetric matrix with spectrum $\Lambda^{-}_{t}$. \\

Similarly, replacing $A-JB-P({\bf e}_{1}X^{\textsuperscript{T}}+{\bf e}_{2}Y^{\textsuperscript{T}})P^{-1}$ in the previous argument by $A-JB+P({\bf e}_{1}X^{\textsuperscript{T}}+{\bf e}_{2}Y^{\textsuperscript{T}})P^{-1}$, which has spectrum $\{a+t+ib,a+t-ib,\lambda_{\frac{n}{2}+3},\ldots,\lambda_{n}\}$, we obtain a nonnegative centrosymmetric matrix with spectrum $\Lambda^{+}_{t}$ (being $\delta=2$).
\end{proof}

\bigskip
\begin{flushleft}
\textbf{Roberto C. D\'{i}az} \\
\texttt{Departamento de Matem\'aticas}, \\
\texttt{Universidad de La Serena} \\
\texttt{Cisternas 1200} \\
\texttt{La Serena, Chile} \\
\texttt{E-mails}: \textit{rdiaz01@ucn.cl}, \ \textit{roberto.diazm@userena.cl}
\end{flushleft}
\bigskip
\begin{flushleft}
\textbf{Ana I. Julio} \\
\texttt{Departamento de Matem\'aticas}, \\
\texttt{Universidad Cat\'olica del Norte} \\
\texttt{Avenida Angamos 0610}, \\
\texttt{Antofagasta, Chile} \\
\texttt{E-mails}: \textit{ajulio@ucn.cl}
\end{flushleft}
\bigskip
\begin{flushleft}
\textbf{Yankis R. Linares} \\
\texttt{Departamento de Matem\'aticas}, \\
\texttt{Universidad Cat\'olica del Norte} \\
\texttt{Avenida Angamos 0610}, \\
\texttt{Antofagasta, Chile} \\
\texttt{E-mails}: \textit{yankis.linares@ce.ucn.cl}
\end{flushleft}
\end{document}